\documentclass{article}
\usepackage{amsmath}
\usepackage{amsfonts}
\usepackage{amsthm}
\usepackage{amssymb}
\usepackage{color}

\newtheorem{theorem}{Theorem}[section]
\newtheorem{lemma}[theorem]{Lemma}
\newtheorem{proposition}[theorem]{Proposition}
\newtheorem{observation}[theorem]{Observation}

\newtheorem{claim}[theorem]{Claim}

\newtheorem{note}[theorem]{Note}

\theoremstyle{definition}
\newtheorem{definition}[theorem]{Definition}
\theoremstyle{remark}
\newtheorem{remark}[theorem]{Remark}

\def\cay{\hskip0.02cm{\rm Cay}\hskip0.01cm}

\def\f2{\mathbb{F}_2}

\def\lip{\hskip0.02cm{\rm Lip}\hskip0.01cm}

\newcommand{\ep}{\varepsilon}

\newcommand{\diam}{{\rm diam}\hskip0.02cm}

\begin{document}

\title{\LARGE Metric characterizations of superreflexivity in terms of word hyperbolic groups and finite graphs}

\author{Mikhail Ostrovskii\footnote{Supported in part by NSF
DMS-1201269.}{ }\footnote{The author would like to thank
Goulnara~Arzhantseva, William~B.~Johnson, Assaf~Naor,
David~Rosenthal, Mark~Sapir, and Gideon~Schechtman for useful
information and discussions related to the subject of this
paper.}\\
\\
Department of Mathematics and Computer Science\\
St. John's University\\
8000 Utopia Parkway\\
Queens, NY 11439\\
USA\\
e-mail: {\tt ostrovsm@stjohns.edu}}

\date{\today}
\maketitle

\begin{large}

\noindent{\bf Abstract.} We show that superreflexivity can be
characterized in terms of bilipschitz embeddability of word
hyperbolic groups. We compare characterizations of
superreflexivity in terms of diamond graphs and binary trees. We
show that there exist sequences of series-parallel graphs of
increasing topological complexity which admit uniformly
bilipschitz embeddings into a Hilbert space, and thus do not
characterize superreflexivity.
\medskip

{\small \noindent{\bf Keywords:} bi-Lipschitz embedding; diamond
graphs; series-parallel graph; superreflexivity; word hyperbolic
group
\medskip

}


\section{Introduction}

The theory of superreflexive Banach spaces was created by James
(see \cite{Jam72a,Jam72b}), an important building block for its
foundation was added by Enflo \cite{Enf72}. One of the equivalent
definitions is: A Banach space $X$ is {\it superreflexive} if and
only if it has an equivalent uniformly convex norm. The theory of
superreflexive spaces is a rich theory, one of the reasons for
this richness is that superreflexive spaces can be characterized
in many different ways, see accounts in \cite{Bea82, BL00, DGZ93,
Dul78, Ost13, Pis11}.\medskip

In addition to having many ``linear'' characterizations, the class
of superreflexive Banach spaces admits many purely metric
characterizations (that is, characterizations which do not refer
to the linear structure of the space). The purpose of this paper
is to obtain some new results on metric characterizations of
superreflexivity. We start by reminding the known metric
characterizations of superreflexivity.
\medskip

The first metric characterization of superreflexivity was
discovered by Bourgain \cite{Bou86}. Recall that a {\it binary
tree of depth $n$} is a finite graph $T_n$ in which each vertex is
represented by a finite (possibly empty) sequence of $0$s and
$1$s, of length at most $n$. Two vertices are adjacent if the
sequence corresponding to one of them is obtained from the
sequence corresponding to the other by adding one term on the
right. We consider a binary tree of depth $n$ as a metric space
whose elements are vertices of the tree, and the distance is the
shortest path distance.

\begin{theorem}[\cite{Bou86}]\label{T:Bourgain} A Banach space $X$ is
nonsuperreflexive if and only if it admits bilipschitz embeddings
with uniformly bounded distortions of finite binary trees of all
depths.
\end{theorem}

In Bourgain's proof the  difficult direction is the ``if''
direction, the ``only if'' is an easy consequence of the theory of
superreflexive spaces. Recently Kloeckner \cite{Klo13+} found a
simple proof of the ``if'' direction.\medskip

We are going to use the following terminology:

\begin{definition}\label{D:TestSp} Let $\mathcal{P}$ be a class of Banach
spaces and let $T=\{T_\alpha\}_{\alpha\in A}$ be a set of metric
spaces. We say that $T$ is a set of {\it test-spaces} for
$\mathcal{P}$ if the following two conditions are equivalent:
\begin{itemize}

\item[{\bf (1)}] $X\notin\mathcal{P}$;

\item[{\bf (2)}] The spaces $\{T_\alpha\}_{\alpha\in A}$ admit
bilipschitz embeddings into $X$ with uniformly boun\-ded
distortions.
\end{itemize}
\end{definition}

\subsection{Versions of Bourgain's characterization with one test-space}

Baudier \cite{Bau07} streng\-thened the ``only if'' part of the
Bourgain theorem and got the following result. Denote by
$T_\infty$ an infinite graph in which each vertex is represented
by a finite (possibly empty) sequence of $0$s and $1$s. Two
vertices are adjacent if the sequence corresponding to one of them
is obtained from the sequence corresponding to the other by adding
one term on the right. We consider $T_\infty$ as a set of vertices
endowed with the shortest path distance.

\begin{theorem}[\cite{Bau07}]\label{T:Baudier} A Banach space $X$ is
nonsuperreflexive if and only if it admits bilipschitz embedding
of $T_\infty$.
\end{theorem}

\begin{remark} There exists a much simpler result stating that there exists a one-test-space characterization
superreflexivity, and even that the test-space can be chosen to be
a tree, see Section \ref{S:OneSpace}. An important feature of
Theorem \ref{T:Baudier} is that the test-space can be chosen to be
$T_\infty$.
\end{remark}

A more general than Theorem \ref{T:Baudier} result was proved in
\cite{Ost12} (recall that a metric space is called {\it locally
finite} if all balls of finite radius in it have finite
cardinalities):

\begin{theorem}[\cite{Ost12}, also
{\cite[Sections 2.1 and 2.3]{Ost13}}]\label{T:FinDet} Let $A$ be a
locally finite metric space whose finite subsets admit uniformly
bilipschitz embeddings into a Banach space $X$. Then $A$ admits a
bilipschitz embedding into $X$.
\end{theorem}

\begin{remark} In \cite[Definition 2.1]{Ost13} locally finite
spaces are required to be uniformly discrete, but this property is
not used in the proof of Theorem \ref{T:FinDet}.
\end{remark}

Theorem \ref{T:FinDet} is more general than Theorem
\ref{T:Baudier} in the sense that combining Theorem \ref{T:FinDet}
with Theorem \ref{T:Bourgain} we immediately get Theorem
\ref{T:Baudier}. The first goal of this paper is to show that any
infinite finitely generated word hyperbolic group (word hyperbolic
in the sense of Gromov \cite{Gro87}) which does not contain
$\mathbb{Z}$ as a finite index subgroup, is also a test-space for
superreflexivity:

\begin{theorem}\label{T:hyperb} {\rm (a)} Let $G$ be a finitely generated word hyperbolic
group. Then the Cayley graph of $G$ admits a bilipschitz embedding
into an arbitrary nonsuperreflexive Banach space.\medskip

{\rm (b)} Let $G$ be an infinite finitely generated word
hyperbolic group which does not have a finite index subgroup
isomorphic to $\mathbb{Z}$. If $G$ admits a bilipschitz embedding
into a Banach space $X$, then $X$ is nonsuperreflexive.
\end{theorem}

In (b) we need to exclude groups which contain $\mathbb{Z}$ as a
finite index subgroup because Cayley graphs of such groups admit
bilipschitz embeddings into $\mathbb{R}$, and hence into any
Banach space of dimension at least $1$ (this fact is well known,
because we do not know a suitable reference we prove it in
Proposition \ref{P:ZFiniteInd} for completeness).

We prove Theorem \ref{T:hyperb} in Section \ref{S:hyperb}.

\subsection{Characterizations of superreflexivity using different
sequences of finite graphs}\label{S:IntroFin}

Johnson and Schechtman \cite{JS09} proved that binary trees in
Theorem \ref{T:Bourgain} may be replaced by some other sequences
of finite graphs, for example, by so-called diamond graphs
introduced in (the conference version of) \cite{GNRS04}. Diamond
graphs can be defined as follows: The {\it diamond graph} of level
$0$ is denoted $D_0$. It has two vertices joined by an edge. $D_i$
is obtained from $D_{i-1}$ as follows. Given an edge $uv\in
E(D_{i-1})$, it is replaced by a quadrilateral $u, a, v, b$. We
endow vertex sets of $D_n$ with their shortest path metrics (each
edge is assumed to have length $1$). In this context we consider
diamonds as finite metric spaces.

\begin{theorem}[\cite{JS09}]\label{T:JS}
A Banach space $X$ is nonsuperreflexive if and only if it admits
bilipschitz embeddings with uniformly bounded distortions of
diamonds $\{D_n\}_{n=1}^\infty$ of all sizes.
\end{theorem}

Our next purpose it to show that this characterization of Johnson
and Schechtman is independent of the Theorem \ref{T:Bourgain} in
the sense that Theorem \ref{T:NETreeInDiam} and the statement
below it hold.\medskip

Let $\{M_n\}_{n=1}^\infty$ and $\{R_n\}_{n=1}^\infty$ be two
sequences of metric spaces. We say that $\{M_n\}_{n=1}^\infty$
admits {\it uniformly bilipschitz embeddings} into
$\{R_n\}_{n=1}^\infty$ if for each $n\in\mathbb{N}$ there is
$m(n)\in\mathbb{N}$ and a bilipschitz map $f_n:M_n\to R_{m(n)}$
such that the distortions of $\{f_n\}$ are uniformly bounded.

\begin{theorem}\label{T:NETreeInDiam}  Binary trees $\{T_n\}_{n=1}^\infty$ do not admit
uniformly bilipschitz embeddings into diamonds
$\{D_n\}_{n=1}^\infty$.
\end{theorem}

The fact that diamonds $\{D_n\}$ do not admit uniformly
bilipschitz embeddings into binary trees $\{T_n\}$ is well known,
it follows immediately from the result of Rabinovich and Raz
\cite[Corollary 5.3]{RR98} stating that the distortion of any
embedding of an $n$-cycle into any tree is $\ge\frac{n}3-1$, and
the fact that $D_n$ $(n\ge 1)$ contains a cycle of length $4^n$
isometrically.

We prove Theorem \ref{T:NETreeInDiam} in Section
\ref{S:NETreeInDiam}.

\begin{remark}\label{R:Laakso} In \cite{JS09} it was proved that a
characterization similar to Theorem \ref{T:JS} can be proved for
the sequence of Laakso graphs. An analogue of Theorem
\ref{T:NETreeInDiam} for Laakso graphs is easy because Laakso
graphs are doubling and binary trees are not, see \cite[Theorem
1.5]{MN13} for an interesting related result.
\end{remark}

Our next result is related to the following remark of Johnson and
Schechtman \cite[Remark 6, p.~188]{JS09}: ``In light of
\cite{BKL07}, it might very well be that Theorem \ref{T:JS}
extends to any series-parallel graph.'' We show that at least in
one direction this is not true, namely, we prove Theorem
\ref{T:SerParL2}. We use the following equivalent definition of
series-parallel graphs (see \cite[p.~243]{GNRS04}).

\begin{definition}\label{D:SerPar} A weighted graph is called {\it
series-parallel} if it can be obtained in the following way:

\begin{itemize}

\item We start with an edge.

\item In each step we add a new vertex and attach it to end
vertices of an already existing edge.

\item At the end of the construction we remove an arbitrary set of
edges.

\end{itemize}

\end{definition}

Series-parallel graphs can be also defined as $K_4$-excluded
graphs, see \cite[Theorem 31.3.7]{DL97}. Another equivalent
definition of series-parallel graphs is in terms of so-called
series and parallel compositions, see \cite{Epp92}; this
definition explains the name.

\begin{theorem}\label{T:SerParL2} There exists a
sequence of series-parallel graphs $\{W_n\}_{n=1}^\infty$ which
satisfies the following conditions.

\begin{enumerate}

\item The underlying unweighted graphs contain
$\{D_n\}_{n=1}^\infty$ as subgraphs.

\item $\{W_n\}_{n=1}^\infty$ admits uniformly bilipschitz
embeddings into $\ell_2$.

\end{enumerate}
\end{theorem}

The point of the first condition is to show that graphs $\{W_n\}$
are topologically complicated (it is clear that such
series-parallel graphs as paths admit uniformly bilipschitz
embeddings into any nonzero Banach space). We prove Theorem
\ref{T:SerParL2} in Section \ref{S:SerPar}.

We do not know whether the second part of the Johnson-Schechtman
remark holds. Namely we do not know whether series-parallel graphs
admit uniformly bilipschitz embeddings into any nonsuperreflexive
space. It is worth mentioning that any counterexample should
involve a non-reflexive Banach space with nontrivial type because
Gupta, Newman, Rabinovich, and Sinclair \cite[Section 4.1]{GNRS04}
proved that series-parallel graphs admit uniformly bilipschitz
embeddings into $\ell_1$.

\section{Characterization of superreflexivity in terms of
hyperbolic groups}\label{S:hyperb}

The purpose of this section is to prove Theorem \ref{T:hyperb}.

\begin{proof} To prove part (a) we are going to use the Buyalo--Dranishnikov--Schroe\-der \cite{BDS07} result on the metric structure of hyperbolic
groups. We need the following definitions. A map $f:X\to Y$
between metric spaces $(X,d_X)$ and $(Y,d_Y)$ is called a {\it
quasi-isometric embedding} if there are $a_1, a_2 > 0$ and $b\ge
0$, such that
\begin{equation}\label{E:QuasiIsom} a_1 d_X(u,v)- b\le d_Y(f(u), f(v))\le a_2d_X(u,v) + b\end{equation} for
all $u,v\in X$. By a {\it binary tree} we mean an infinite tree in
which each vertex has degree $3$. By a {\it product of trees},
denoted $(\oplus_{i=1}^n T(i))_1$, we mean (see \cite[Remark
12.1.2]{BS07}) their Cartesian product with the $\ell_1$-metric,
that is,
\begin{equation}\label{E:L_1Prod}d(\{u_i\},\{v_i\})= \sum_{i=1}^n
d_{T(i)}(u_i,v_i).\end{equation}

\begin{theorem}[\cite{BDS07}]\label{T:BDS07} Every Gromov hyperbolic group admits a
quasi-isometric embedding into the product of finitely many copies
of the binary tree.
\end{theorem}

\medskip

This result shows that in order to prove part (a) of Theorem
\ref{T:hyperb} it suffices to find

\begin{enumerate}

\item \label{I:Bilip} A bilipschitz embedding of a product of
binary trees in the sense \eqref{E:L_1Prod} into any
nonsuperreflexive space.

\item \label{I:CorrQuasiIs} A way to modify a quasi-isometric
embedding of a hyperbolic group into a Banach space in order to
get a bilipschitz embedding.
\end{enumerate}

To achieve the first goal we use Theorem \ref{T:FinDet} stating
that embeddability of a locally finite metric space into a Banach
space is finitely determined. It is clear that the product (in the
sense \eqref{E:L_1Prod}) of finitely many binary trees is a
locally finite metric space. Therefore to achieve the goal of item
\ref{I:Bilip} it suffices to find bilipschitz embeddings of
products of $n$ binary trees of depth $k$ into any
nonsuperreflexive space $X$ with distortions bounded independently
of $k$ (for a fixed $n$).
\medskip

Bourgain \cite[Section 3]{Bou86} showed that there is an absolute
constant $B$ such that for any $k$ and any nonsuperreflexive
Banach space $X$ there is a $B$-bilipschitz embedding of a binary
tree of depth $k$ into $X$ (see \cite[Section 9.1]{Pis11} for a
more detailed description of this embedding). To embed a finite
sum of finite binary trees $(\oplus_{i=1}^n T(i))_1$ (where each
$T(i)$ is a binary tree of depth $k$) into an arbitrary
nonsuperreflexive Banach space $X$ we do the following. First we
find a finite-dimensional subspace $X_1\subset X$ containing a
$B$-bilipschitz image of $T(1)$ (where $B$ is the absolute
constant mentioned above). We assume that this embedding, and the
embeddings of other $T(i)$ introduced below have Lipschitz
constants $\le B$, and their inverses  have Lipschitz constants
$\le 1$.

\begin{remark} Using the argument of James \cite{Jam64} (see
\cite[p.~261]{Bea82}, \cite{Jam72a}, and \cite{SS70}) one can show
that $B$ can be chosen to be any constant in $(1,\infty)$.
\end{remark}

Let $\lambda\in(0,1]$. Recall that a subspace $M\subset X^*$  is
called {\it $\lambda$-norming over a subspace $Y\subset X$} if
\[\forall y\in Y~\sup\{|f(y)|:~f\in M,~||f||\le 1\}\ge\lambda||y||.\]

Pick any $\lambda\in(0,1)$. Let $M_1\subset X^*$ be a
finite-dimensional subspace which is $\lambda$-norming over $X_1$.
The existence of such subspace can be shown as follows. Let
$\{x_i\}_{i=1}^n$ be an $(1-\lambda)$-net in the unit sphere of
$X_1$ and let $M_1$ be the linear span of functionals $x_i^*$
satisfying the conditions $||x_i^*||=1$ and $x_i^*(x_i)=1$. The
verification that $M_1$ is $\lambda$-norming is immediate.
\medskip

The subspace $(M_1)_{\top}:=\{x\in X:~\forall x^*\in M_1~
x^*(x)=0\}$ is of finite codimension. It is clear that for $x_1\in
X_1$ and $x_2\in(M_1)_{\top}$ we have
$||x_1+x_2||\ge\lambda||x_1||$.\medskip

It is also clear that the subspace $(M_1)_{\top}$ is
nonsuperreflexive. Hence we can find in it a finite-dimensional
subspace $X_2$ containing a $B$-bilipschitz copy of $T(2)$. Now
let $M_2\subset X^*$ be a finite-dimensional subspace which is
$\lambda$-norming over  the linear span of $X_1\cup X_2$. We
continue in an obvious way (e.g. we let $X_3$ to be a
finite-dimensional subspace of  $(M_2)_{\top}$ containing a
$B$-bilipschitz image of $T(3)$).
\medskip

The finite-dimensional subspaces $\{X_i\}_{i=1}^n$ constructed in
this manner are such that for any choice of $x_i\in X_i$,
$i=1,\dots, n$, the inequalities
\begin{equation}\label{E:norming}\begin{split} \lambda\max\{||x_1||,
||x_1+x_2||,\dots,||x_1+\dots+x_{n-1}||\}&\le
||x_1+\dots+x_n||\\&\le||x_1||+||x_2||+\dots+||x_n||\end{split}\end{equation}
hold. The leftmost inequality implies that for each $1\le t<n$ we
have
\[\begin{split}||x_{t+1}+\dots+x_n||&\le ||x_1+\dots+x_n||+||x_1+\dots+x_t||\\
&\le\frac{\lambda+1}{\lambda}\,||x_1+\dots+x_n||.
\end{split}\]
Applying \eqref{E:norming} again we get
\[||x_{t+1}||\le \frac{\lambda+1}{\lambda^2}||x_1+\dots+x_n||.\]
We conclude that
\[||x_1+\dots+x_n||\ge\frac{\lambda^2}{\lambda+1}\max_{1\le j\le n}||x_j||\ge \frac{\lambda^2}{(\lambda+1)n}\sum_{i=1}^n||x_i||.\]
So the $\ell_1$-sum $(\oplus_{i=1}^kX_i)_1$ admits a natural
bilipschitz embedding into $X$ with distortion $\le
\frac{(\lambda+1)n}{\lambda^2}$. The image of this embedding is
the linear span of the spaces $X_i$ in $X$. Combining this
embedding with the $B$-bilipschitz embeddings of $T(i)$ into $X_i$
we get a $\frac{B(\lambda+1)n}{\lambda^2}$-bilipschitz embedding
of $(\oplus_{i=1}^n T(i))_1$ into $X$. Applying Theorem
\ref{T:FinDet} we get that a finite product of (infinite) binary
trees admits a bilipschitz embedding into $X$. Combining this with
Theorem \ref{T:BDS07}, we get that for an arbitrary finitely
generated word hyperbolic group $G$ there is a quasi-isometric
embedding $T$ of $G$ into $X$.
\medskip

It remains to modify the quasi-isometric embedding in order to get
a bilipschitz embedding. Easy examples show that this cannot be
done for general metric spaces; our purpose is to show that this
can be done in the case where we embed a locally finite uniformly
discrete metric space into an infinite-dimensional Banach space
(observe that a finite-dimensional Banach space cannot be
nonsuperreflexive).

\begin{lemma}\label{L:QuasiBilip} Let $M$ be a locally finite uniformly discrete metric space admitting a
quasi-isometric embedding into an infinite dimensional Banach
space $X$. Then $M$ admits a bilipschitz embedding into $X$.
\end{lemma}

\begin{proof} We may assume that the distance between any two distinct elements in $M$ is at least
$1$ (multiplying all distances in $M$ by some positive number, if
necessary). Let $T:M\to X$ be a quasi-isometric embedding with
constants $a_1$, $a_2$, and $b$ (see \eqref{E:QuasiIsom}). Since
the distance between any distinct elements of the $M$ is at least
$1$, it is clear that $T$ is $(a_2+b)$-Lipschitz.\medskip

To get the desired estimate from below we perturb the map in order
to make it injective and to make its image uniformly discrete.
This is done in the following way. We consider a $1$-net $N$ in
$X$, that is, a set of points $\{x_i\}_{i=1}^\infty$  such that
$||x_i-x_j||\ge 1$ for $i\ne j$ and $\forall x\in X\,\exists i\in
\mathbb{N}$ $||x-x_i||\le 1$ (without loss of generality we may
assume that $X$ is separable since the linear span of $T(M)$ is
separable). Since $X$ is infinite-dimensional, it is easy to see
that the ball of radius $4$ centered at any $x\in X$ contains
infinitely many points of $N$. In fact, one can show that a ball
of radius $3$ centered at $x$ contains an infinite sequence
$\{s_i\}_{i=1}^\infty$ satisfying $||s_i-s_j||\ge \frac52$. This
implies that points of $N$ which are close to $\{s_i\}$ are
distinct and are inside the ball of radius $4$ centered at $x$.
(Using results of Kottman \cite{Kot75} and  Elton-Odell
\cite{EO81} one can improve the constant $4$ in this statement,
but we do not need this improvement here). On the other hand,
$T(M)$, as an image of a locally finite metric space under a
quasi-isometric embedding, is locally finite (recall that we do
not require a locally finite metric space to be uniformly
discrete), and each point in it has at most finitely many
pre-images. Therefore for each $u\in M$ we can find a point
$R(u)\in N$ such that $||T(u)-R(u)||\le 4$ and $||R(u)-R(v)||\ge
1$ if $u\ne v$, $u,v\in M$. It is easy to verify that $R$ is also
a quasi-isometric embedding (\eqref{E:QuasiIsom} is satisfied with
the same $a_1$ and $a_2$ and with $b'=b+8$), hence it is also a
Lipschitz map.
\medskip

To show that it is bilipschitz it remains to get an estimate for
$||R(u)-R(v)||$ from below. We observe that
\begin{equation}\label{E:QuasiIsomR} a_1 d_M(u,v)- b'\le ||R(u)-R(v)||\end{equation}
implies the estimate $||R(u)-R(v)||\ge \frac{a_1}2\,d_M(u,v)$ if
$d_M(u,v)\ge \frac{2b'}{a_1}$. If $b'=0$, there is nothing to
prove. If $b'>0$ and $d_M(u,v)\le \frac{2b'}{a_1}$, we have
$||R(u)-R(v)||\ge \frac{a_1}{2b'}\,d_M(u,v)$ just because
$||R(u)-R(v)||\ge 1$. Thus $R$ is a bilipschitz
embedding.\end{proof}

Now we turn to the proof of part (b). To prove this result we need
another result on the structure of hyperbolic groups. We use the
metric geometry terminology of \cite{BH99}. Let $K\ge 0$. A subset
$Y$ of a geodesic metric space $X$ is called {\it $K$-quasiconvex}
if any geodesic path in $X$ with endpoints in $Y$ lies in the
$K$-neighborhood of $Y$ . A subset $Y$ is called {\it quasiconvex}
if it is $K$-quasiconvex for some $K<\infty$.

\begin{theorem}\label{T:FreeQuasiC} Let $G$ be a finitely generated word hyperbolic group
which does not contain an infinite cyclic subgroup of finite
index. Then $G$ contains a free subgroup $H$ with two generators
as a quasiconvex subgroup.
\end{theorem}

In the case of torsion free groups this result was obtained by
Arzhantseva \cite{Arz01}. The general case follows by combining
results of Dahmani, Guirardel, Osin \cite[Theorem 6.14]{DGO11} and
Sisto \cite[Theorem 2]{Sis13}.

A proof of part (b) of Theorem \ref{T:hyperb} can be derived from
Theorem \ref{T:FreeQuasiC} as follows. Suppose that a Banach space
$X$ admits a bilipschitz embedding of an infinite word hyperbolic
group $G$ satisfying the conditions of Theorem \ref{T:FreeQuasiC}.
By Theorem \ref{T:FreeQuasiC}, $G$ contains a free subgroup $H$
with two generators, we denote them $a$ and $b$, as a quasiconvex
subgroup. The quasiconvexity implies that the identical map of $H$
endowed with the metric of the graph
$\cay(H,\{a,b,a^{-1},b^{-1}\})$ onto $H$ endowed with the metric
induced from $G$ is a quasi-isometric embedding, see \cite[Lemma
3.5]{BH99}.
\medskip

Thus we get a quasi-isometric embedding of the tree
$\cay(H,\{a,b,a^{-1},b^{-1}\})$  into the Banach space $X$. It is
easy to see that $X$ cannot be finite-dimensional. Using Lemma
\ref{L:QuasiBilip} we modify the quasi-isometric embedding and get
a bilipschitz embedding of the tree
$\cay(H,\{a,b,a^{-1},b^{-1}\})$ into $X$. By the result of
Bourgain \cite{Bou86} (see also a short proof in \cite{Klo13+}),
we get that the Banach space $X$ is nonsuperreflexive.
\end{proof}

The proof of the next proposition is well known and elementary. It
is included because I have not found a suitable reference. We
would like to emphasize that in Proposition \ref{P:ZFiniteInd} we
consider the Cayley graph as a set of elements of $G$ with the
shortest path distance (and not as a $1$-dimensional simplicial
complex).

\begin{proposition}\label{P:ZFiniteInd} Let $G$ be  a group
containing $\mathbb{Z}$ as a finite index subgroup. Then the
Cayley graph $\cay(G,S)$ with respect to any finite set of
generators admits a bilipschitz embedding into $\mathbb{R}$.
\end{proposition}

\begin{proof} Let $x$ be the generator
of $\mathbb{Z}$. We use multiplicative notation, so
\[\mathbb{Z}=\{\dots,x^{-n},\dots,x^{-1},e,x,\dots,x^n,\dots\}.\]
Let $g_1=e,g_2,\dots,g_{m}$ be representatives of left cosets of
$\mathbb{Z}$ in $G$. We consider the most straightforward
embedding. We map
\[\dots,x^{-n},\dots,x^{-1},e,x,\dots,x^n,\dots\]
to
\[\dots,-n,\dots,-1,0,1,\dots,n,\dots,\]
respectively; and map
\[x^n,g_2x^n,\dots,g_{m}x^{n}\]
to points
\[n,n+\frac1m,\dots,n+\frac{m-1}m,\]
respectively

We consider the following generating set
$S=\{x,x^{-1},g_2,g_2^{-1},\dots,g_m,g_m^{-1}\}$ and let
$\cay(G,S)$ be the right-invariant Cayley graph.  Since word
metrics for different sets of generators are
bilipschitz-equivalent, it suffices to show that the described
above embedding, let us denote it $\varphi$, is bilipschitz as an
embedding of $\cay(G,S)$ into $\mathbb{R}$.
\medskip

To estimate the Lipschitz constant of
$\varphi:\cay(G,S)\to\mathbb{R}$ we need to estimate from above
the distance between images of adjacent vertices, that is,
$|\varphi(u)-\varphi(x^{\pm1}u)|$ and
$|\varphi(u)-\varphi(g_i^{\pm1}u)|$, $i=2,\dots,m$.\medskip

We have $u=g_ix^k$ for some $i$ and $k$. Observe that
$g_jg_i=g_{k(i,j)}x^{p(i,j)}$, $g_j^{-1}g_i=g_{l(i,j)}x^{r(i,j)}$,
$xg_i=g_{k(i)}x^{p(i)}$, and $x^{-1}g_i=g_{l(i)}x^{r(i)}$. This
implies that $x^{\pm1}u$ and $g_i^{\pm1}u$ are of the form
$g_sx^{k+t}$, where the absolute value of $t$ is bounded above by
\[T:=\max_{i,j}\{|p(i,j)|, |r(i,j)|, |p(i)|, |r(i)|\}.\]

The number $T$ depends on the group $G$ and the choice of the
coset representatives, but not on $i$ and $k$. The desired
estimates of $|\varphi(u)-\varphi(x^{\pm1}u)|$ and
$|\varphi(u)-\varphi(g_i^{\pm1}u)|$, $i=2,\dots,m$ from above
follow.
\medskip

Now we estimate the Lipschitz constant of $\varphi^{-1}$. If
$k=l$, we have $|\varphi(g_ix^k)-\varphi(g_jx^l)|\ge\frac1m$, and
the word distance between $g_ix^k$ and $g_jx^k$ is $2$, and we are
done in this case. Observe that for $k>l$ we have
$|\varphi(g_ix^k)-\varphi(g_jx^l)|\ge
\left|k-l-\frac{m-1}m\right|$. On the other hand, we can reach
$g_ix^k$ from $g_jx^l$ traversing $k-l+2$ edges (first we traverse
the edge corresponding to $g_j^{-1}$,  then edges corresponding to
$x$ ($(k-l)$ times since $k>l$), then $g_i$). The estimate for the
Lipschitz constant of $\varphi^{-1}$ follows.
\end{proof}

\section{Binary trees do not admit uniformly bilipschitz embeddings into
diamonds}\label{S:NETreeInDiam}

\begin{proof}[Proof of Theorem
\ref{T:NETreeInDiam}] We are going to show that for any
$k\in\mathbb{N}$, no matter how we choose the numbers
$m(n)\in\mathbb{N}$ and $p(n)\in\mathbb{N}\cup\{0\}$, it is
impossible to find maps $F_n: T_n\to D_{m(n)}$ such that
\begin{equation}\label{E:bilip}\forall n~\forall u,v\in
V(T_n)\quad 2^{p(n)}d_{T_n}(u,v)\le d_{D_{m(n)}}(F_nu,F_nv)\le
2^k\cdot 2^{p(n)}d_{T_n}(u,v).\end{equation}

Let us remind that we normalize distances in diamonds in such a
way that each edge has length $1$.\medskip

{\bf Note:} We do not have to consider negative $p(n)$ because in
such cases we may replace $D_{m(n)}$ by $D_{m(n)-p(n)}$ and use
the natural map of $D_{m(n)}$ into $D_{m(n)-p(n)}$, which
multiplies all distances by $2^{-p(n)}$.\medskip

We are going to use the notion of a {\it subdiamond}. It is
defined as a part of a diamond which evolved from  an edge. A
subdiamond has naturally defined {\it top} and {\it bottom}.
(These notions are defined up to the choice of the bottom and the
top of $D_0$.) The {\it height} of a subdiamond is the distance
from its top to its bottom.

\begin{lemma}\label{L:Entropy} The cardinality of a $2^{p(n)}$-separated set (i.e. a set satisfying $d(u,v)\ge 2^{p(n)}$ for any $u\ne v$) in a subdiamond of height $2^h$
does not exceed $2\cdot 4^{h-p(n)}$.
\end{lemma}

\begin{proof} It is easy to see that  each subdiamond of height
$2^{p(n)}$ contains at most two vertices out of each
$2^{p(n)}$-separated set. The number of subdiamonds of height
$2^{p(n)}$ in a diamond of height $2^h$ is equal to the number of
edges in the diamond of height $2^{h-p(n)}$. This number of edges
is $4^{h-p(n)}$, because in each step of the construction of
diamonds the number of edges quadruples.
\end{proof}

By a {\it subtree} in $T_n$ we mean the subgraph consisting of
some vertex and all of its descendants. By Lemma \ref{L:Entropy},
if there is a subtree $S$ of $T_n$ whose order (number of
vertices) is $>2\cdot 4^{h-p(n)}$ and whose root is mapped by
$F_n$ into a subdiamond $M$ of height $2^h$, some of the vertices
of $S$ have to be mapped outside the subdiamond, and so some of
the root-leaf paths in $S$ have to {\it leave} $M$.
\medskip

Our next observation is: there are two {\it exits} in a
subdiamond. We mean that there are two vertices in any subdiamond
$M$ whose deletion would separate the subdiamond from the rest of
the diamond. This happens just because the subdiamond evolved from
an edge. In the proof of Lemma \ref{L:Only2Leave} we combine this
fact with the bilipschitz condition on $F_n$ and get the
conclusion that two root-leaf paths in $S$ with small intersection
cannot leave $M$ through the same exit if they stay in $M$ for
long time before leaving. Here we say that a path $P$ in $T_n$
{\it leaves $M$ through the exit} $v$ if there are two consecutive
vertices $u$ and $w$ in $P$, such that $F_nu\in M$, $F_nw\notin M$
and
\begin{equation}\label{E:leave} d_{D_{m(n)}}(F_nu,v)+ d_{D_{m(n)}}(v,F_nw)\le
2^k\cdot2^{p(n)}.\end{equation} Observe that the bilipschitz
condition \eqref{E:bilip} implies that if some vertices of $P$ are
in $M$, and some other vertices are not, then $P$ should leave $M$
through one of the two exits.

\begin{lemma}\label{L:Only2Leave} Let $S$ be a subtree of $T_n$ whose root $s$ is mapped into a
subdiamond $M$. Suppose that $F_n$ maps at least $2^{k}+2$
generations of descendants of $s$ into $M$. Let $s_1,s_2,s_3$, and
$s_4$ be $4$ grandchildren of $s$, denote by $S_1,S_2,S_3$, and
$S_4$ the sets of those descendants of $s_1,s_2,s_3$, and $s_4$,
respectively, which belong to generations $2^k+3$ and later ones,
where we mean that $s$ is generation $0$. Then only two out of the
four sets $S_1,S_2,S_3$, and $S_4$ can have vertices whose
$F_n$-images are not in $M$.
\end{lemma}

\begin{proof} It is easy to check that the pairwise distances between $F_n$-images of $S_1$, $S_2$,
$S_3$, and $S_4$ in $D_{m(n)}$ are  $\ge 2^{p(n)}(2^{k+1}+4)$.
Suppose that at least three of the sets $S_1, S_2, S_3, S_4$
contain vertices whose $F_n$-images are not in $M$. We may assume
that the vertices are $i_1\in S_1$, $i_2\in S_2$, and $i_3\in
S_3$. Consider paths $\{P_i\}_{i=1}^3$ in $S$ joining these
vertices and $s$. Each of these paths leaves $M$. Since there are
two exits and three paths, two of the paths should leave $M$
through the same exit $v$. Since the paths cannot leave $M$ within
the first $2^k+2$ generations from $s$,  each of these two $P_i$
contains consecutive vertices $u_i$ and $w_i$ satisfying
\eqref{E:leave} such that $F_n(u_i)\in M$, $F_n(w_i)\notin M$, and
$u_i$ are in the generation at least $2^k+2$ from $s$ and $w_i$
are in the generation at least $2^k+3$. It is clear that this
implies the existence of two vertices, $t$ and $s$, among these
$u_i$ and $w_i$ such that $d_{T_n}(t,s)\ge 2^{k+1}+2$ and
$d_{D_{m(n)}}(F_nt,F_ns)\le
d_{D_{m(n)}}(F_nt,v)+d_{D_{m(n)}}(v,F_ns)\le 2^k2^{p(n)}$. This
contradicts the bilipschitz condition \eqref{E:bilip}.
\end{proof}

Lemma \ref{L:Only2Leave} suggests the following plan for
completion of the proof. We find a subdiamond $M$ with height
$2^h$ in $D_{m(n)}$ which contains all of the $F_n$-images of the
first $2^{k}+2$ generations of a rooted in $s$ subtree $S$ (of
$T_n$), and the total number of generations in $T_n$ which contain
descendants of $s$ is more than $2(h-p(n)+1)$.

In fact, if we find such an $M$, by Lemma \ref{L:Only2Leave}, more
that half of the descendants of $s$ should have their $F_n$-images
in $M$. The number of such descendants is $> 2\cdot 4^{h-p(n)}$,
which is more than the possible number of $2^{p(n)}$-separated
points in $M$ estimated by Lemma \ref{L:Entropy}, and so this
would complete the proof.
\medskip

So it remains to find a suitable subdiamond $M$. The reasons for
which it is possible is that,  on one hand, vertices which appear
later (in diamond's construction) are ``dense'' in the diamond
and, on the other hand, they have neighborhoods which are
contained in subdiamonds of controlled size. We present details
below.\medskip

Now we introduce generations of vertices in a diamond. We label
them from the end. Generation number $1$ is the set of vertices
appeared in the last step of the construction of $D_{m(n)}$.
Generation number $2$ is the set of vertices appeared in the
previous step of the construction, so on, there are $m(n)$
generations (two original vertices do not belong to any of the
generations.) The following is clear from the construction:

\begin{observation}\label{O:Generat} {\bf (1)} Let $v$ be a vertex of generation number $r$,
$r\in\{1,\dots,m(n)\}$. Then the $2^{r-1}$-neighborhood of $v$ is
contained in a subdiamond of height $2^r$.\medskip

\noindent{\bf (2)} Let $Z_r$ be the set of all vertices of
generation number $r$. Then the connected components of
$D_{m(n)}\backslash Z_r$ have diameters $<2^{r}$.
\end{observation}

Now we consider the binary tree $T_n$ of depth $n=L(k)$, where
$L(k)$ is a ``large'' number depending only on $k$, we shall
specify our choice of $L(k)$ later.

Let $q=q(k)\in \mathbb{N}$ ($q(k)$ also will be chosen later). No
matter how we choose the image of the root of $T_n$ in $D_{m(n)}$,
by Observation \ref{O:Generat} {\bf (2)}, within $2^{q}$ steps,
following $F_n$-images of any of the descending paths in $T_n$, we
shall ``pass over'' (recall that we make ``steps'' of lengths
between $2^{p(n)}$ and $2^{k+p(n)}$ for each edge) a vertex $z$
belonging to the generation $Z_{q+p(n)}$, we can ``miss'' $z$ by
$2^{k+p(n)-1}$, let $s$ be the vertex in a generation with number
$\le 2^q$ of $T_n$ such that $d_{D_{m(n)}}(F_n(s),z)\le
2^{k+p(n)-1}$. After that we consider the subtree of descendants
of $s$ for $2^{k}+2$ generations (as we did in Lemma
\ref{L:Only2Leave}). All of them are in the
$2^{p(n)}\cdot(2^{k-1}+(2^{k}+2)2^k)$-neighborhood of $z$. Now we
pick $q=q(k)$ in such a way that
\[2^{q(k)-1}\ge 2^{k-1}+(2^k+2)2^k.\]
Then, by Observation \ref{O:Generat} {\bf (1)}, all of the first
$2^{k}+2$ generations of descendants of $s$ (including $s$) are
mapped into a subdiamond $M$ of height $2^{q(k)+p(n)}$. Now we
pick $L(k)>2^{q(k)}+2(q(k)+1)$. With this choice, the subdiamond
$M$ and the vertex $s$ have the desired properties. We mean that
the total number of generations of $T_n=T_{L(k)}$ which contain
descendants of $s$ is more than $2(q(k)+1)=2((q(k)+p(n))-p(n)+1)$,
see the discussion following Lemma \ref{L:Only2Leave}.
\end{proof}

\section{Example of a series-parallel family admitting uniformly
bilipschitz embeddings into $\ell_2$}\label{S:SerPar}

\begin{proof}[Proof of Theorem \ref{T:SerParL2}] The graphs which we use for this construction are
close to diamond graphs (see the beginning of Section
\ref{S:IntroFin}), one of the differences is that we do not remove
edges of graphs which appear earlier in the sequence. Also the
metric is obtained by introducing weights of edges and the
corresponding shortest weighted path distance (subdividing edges
one can avoid using weights).\medskip

We pick a number $\ep\in\left(0,\frac12\right)$. The sequence
$\{W_n\}_{n=0}^\infty$ of {\it weighted diamonds} is defined in
terms of diamonds $\{D_n\}_{n=0}^\infty$ as follows:

\begin{itemize}

\item $W_0$ is the same as $D_0$

\item $W_1=D_1\cup W_0$ with edges of $D_1$ given weights
$\left(\frac12+\ep\right)$; weight of the edge of $W_0$ stays as
$1$ (as it was in the first step of the construction).

\item $W_2=D_2\cup W_1$ with edges  of $D_2$ given weights
$\left(\frac12+\ep\right)^2$; weights of the edges of $W_{1}$ stay
as they were in the previous step of the construction.

\item \dots.

\item $W_n=D_n\cup W_{n-1}$ with edges of $D_n$ given weights
$\left(\frac12+\ep\right)^n$; weights of the edges of  $W_{n-1}$
stay as they were in the previous step of the construction.

\end{itemize}

We consider $\{W_n\}$ as finite metric spaces whose elements are
vertices and the distance between two vertices is the weighted
length of the shortest path. We define embeddings
$F_n:W_n\to\ell_2$ as follows:

\begin{itemize}

\item The map $F_0$ maps the vertices of $D_0$ to $0$ and $e_0$,
respectively (where $\{e_i\}_{i=0}^\infty$ is the unit vector
basis of $\ell_2$). It is clear that $F_0$ is an isometric
embedding.

\item The map $F_{n}$, $n\ge1$, is an extension of the map
$F_{n-1}$. The description is generic for all $n\ge 1$. Each
vertex $w$ of $W_n\backslash W_{n-1}$ corresponds to two vertices
of $W_{n-1}$: $w$ is the vertex of the $2$-edge path joining $u$
and $v$. We map the vertex $w$ to
$\frac12(F_{n-1}(u)+F_{n-1}(v))\pm\omega_n e_{uv}$, where $e_{uv}$
is an element of $\{e_i\}_{i=1}^\infty$ picked for the edge $uv$
(we pick different $e_i$ for different edges) and $\pm$ are picked
differently for the pair of vertices $w,\widetilde w$ which
corresponds to the same edge $uv$ (our construction is such that
there are two such vertices); and $\omega_n$ is picked in such a
way that $||F_n(u)-F_n(w)||=d_{W_n}(u,w)$, so
$\omega_n=\sqrt{\ep+\ep^2}\left(\frac12+\ep\right)^{n-1}$.

\end{itemize}

Now we estimate the distortion. Observe that $F_n$ is
distance-preserving on edges. Therefore $\lip(F_n)\le 1$, and we
need to estimate the Lipschitz constant of $F_n^{-1}$ only.
\medskip

Let us start with $\lip(F_1^{-1})$, it is attained on the only
pair of vertices of $W_1$ which is not an edge, and it is
therefore
\[\lip(F_1^{-1})=\frac{1+2\ep}{2\sqrt{\ep+\ep^2}}.
\]

The estimate of $\lip(F_n^{-1})$, $n\ge 2$, can be done in general
in the following way. We consider any {\bf shortest} path between
two vertices in $W_n$.

\begin{note} Some caution is needed for arguments about shortest paths in $W_m$, because they can be quite different from shortest paths
in diamonds. For example, if $m$ is such that
\begin{equation}\label{E:ShrtUsDiag} \left(\frac12+\ep\right)+\left(\frac12+\ep\right)^2+\dots+\left(\frac12+\ep\right)^m\ge
1+\left(\frac12+\ep\right)^m\end{equation} (such number $m$
obviously exists if $\ep>0$), then a shortest path between two
vertices in $W_m$ can consist of the edge of $D_0$ and one more
edge of length $\left(\frac12+\ep\right)^m$.
\end{note}

\begin{claim}\label{C:le2} A shortest path between two vertices in $W_n$ can contain edges of each possible
length: \[1,\left(\frac12+\ep\right), \left(\frac12+\ep\right)^2,
\left(\frac12+\ep\right)^3, \dots\] at most twice, actually for
$1$ this can happen only once because there is only one such edge.
\end{claim}

\begin{proof} Let $e$ be one of the longest edges
in the path and $\left(\frac12+\ep\right)^k$ be its length. As for
diamonds, we define weighted subdiamonds as subsets evolving from
edges (as sets of vertices they coincide with the subdiamonds
defined before). The edge from which a subdiamond evolved is
called its {\it diagonal}. Consider the subdiamond $S$ containing
$e$ with diagonal of length $\left(\frac12+\ep\right)^{k-1}$ (here
we assume that $k\ne 0$, for $k=0$ the statement is trivial). Let
$e=uv$, without loss of generality we may assume that $u$ is the
bottom of $S$ (turning the graph upside down, if needed).
\medskip

The rest of the path consists of two pieces: (1) The one which
starts at $v$; (2) The one which starts at $u$. We claim that the
part which starts at $v$ can never leave $S$. It obviously cannot
leave through $u$, it cannot leave through the top of $S$, let us
denote it $t$, because otherwise the piece of the path between $u$
and $t$ could be replaced by the diagonal of $S$, which is
strictly shorter.\medskip

This implies that the part of the path in $S$ can contain edges
only shorter than $\left(\frac12+\ep\right)^k$ (except $e$). For
the next edge in this part of the path we can repeat the argument
and get (by induction) that lengths of edges in the remainder of
the path in $S$ are strictly decreasing.\medskip

The part of the path which starts at $u$ can be considered
similarly.
\end{proof}

Now we need to analyze the ``directions'' in which the path can
go. Let $e=uv$, the subdiamond $S$, $t$, and the two parts of the
shortest path be as above.\medskip

If the part of the path which starts a $u$ does not enter $S$, the
estimate of the Lipschitz constant can be based just on the part
contained in $S$. Here is the estimate:

\begin{itemize}

\item By Claim \ref{C:le2} we get that the length of the path is
$\displaystyle{\le
2\left(\frac12+\ep\right)^k\cdot\frac1{\frac12-\ep}}$

\item On the other hand the $e_{ut}$-coordinate of the path

\begin{itemize}

\item[(a)] Does not change along the path which starts at $u$ and
leaves $S$.

\item[(b)] Changes by
$\left(\frac12+\ep\right)^{k-1}\sqrt{\ep+\ep^2}$ when we traverse
$e$.

\item[(c)] Cannot change by more than
\begin{equation}\label{E:MaxChange}\left(1-\left(\frac12\right)^m\right)\left(\frac12+\ep\right)^{k-1}\sqrt{\ep+\ep^2}\end{equation}
on the path which starts at $v$, where $m\in \mathbb{N}$ is the
least number satisfying \eqref{E:ShrtUsDiag}.

\end{itemize}

Only the statement (c) requires a proof. We observe that
traversing the image of an edge of length
$\left(\frac12+\ep\right)^{k+d}$ $(d=0,1,2,\dots)$ of this path we
change the $e_{ut}$ coordinate by
$\left(\frac12\right)^d\left(\frac12+\ep\right)^{k-1}\sqrt{\ep+\ep^2}$
(this follows by induction from our definition of $F_n$).
Therefore, if the change exceeds \eqref{E:MaxChange}, it implies
that the path staring at $v$ contains edges of all lengths
\[\left(\frac12+\ep\right)^{k+1},\left(\frac12+\ep\right)^{k+2},\dots,\left(\frac12+\ep\right)^{k+m}.\]
But then, because of \eqref{E:ShrtUsDiag}, it is not a shortest
path between its ends (using the diagonal we get a shorter path).
\end{itemize}

The same argument works in the case where the path starting at $u$
stays inside $S$, because in this case (as is easy to see) the
sign of the $e_{ut}$-coordinate on this part of the path is
different from its sign on the first part of the path.
\medskip

Therefore in both cases we get that the quotient of (the length of
the path)/(the distance between the images) does not exceed
\[\frac{2^{m+1}}{\left(\frac12-\ep\right)\sqrt{\ep+\ep^2}}\left(\frac12+\ep\right).\]
Since $m$ depends only on the choice of $\ep$ (and not on $n$), we
get the desired estimate for $\sup_n\lip(F_n^{-1})$.
\end{proof}

\section{Characterizations using one test-space}\label{S:OneSpace}

Our purpose is to show that a class of Banach spaces for which
there is a sequence of finite test-spaces has one test-space,
similar results hold for test-spaces satisfying some additional
conditions. The proof is very simple (it simplifies some of the
results of \cite{Ost13b}).

\begin{proposition} {\bf (a)} Let $\{S_n\}_{n=1}^\infty$ be a sequence of finite test-spaces for some class $\mathcal{P}$ of Banach
spaces containing all finite-dimensional Banach spaces. Then there
is a metric space $S$ which is a test-space for $\mathcal{P}$.
\medskip

{\bf (b)} If $\{S_n\}_{n=1}^\infty$ are

\begin{itemize}

\item unweighted graphs,

\item trees,

\item graphs with uniformly bounded degrees,

\end{itemize}
then $S$ also can be required to have the same property.
\end{proposition}

\begin{proof} In all cases the space $S$ will contain subspaces
isometric to each of $\{S_n\}_{n=1}^\infty$. Therefore the only
implication which is nontrivial is that the embeddability of
$\{S_n\}_{n=1}^\infty$ implies the embeddability of $S$.
\medskip

Each finite metric space can be considered as a weighted graph
with it shortest path distance. In all cases we construct the
space $S$ as an infinite graph by joining $S_n$ with $S_{n+1}$
with a path $P_n$ whose length is $\ge\max\{\diam S_n, \diam
S_{n+1}\}$. To be more specific, we pick in each $S_n$ a vertex
$O_n$ and let each of the paths mentioned above be a path joining
$O_n$ with $O_{n+1}$. We endow the infinite graph $S$ with the
shortest path distance. It is clear that all of the conditions are
satisfied. It remains only to show that each infinite-dimensional
Banach space which admits bilipschitz embeddings of
$\{S_n\}_{n=1}^\infty$ with uniformly bounded distortions admits a
bilipschitz embedding of $S$.
\medskip

So let $X$ be an infinite dimensional Banach space admitting
bilipschitz embeddings of spaces $\{S_n\}_{n=1}^\infty$ with
uniformly bounded distortions. Let $F_n:S_n\to X$ be the
corresponding embeddings. We may assume that $\lip(F_n)=1$ for
each $n$ and $\lip((F_n)^{-1})$ are uniformly bounded. Since the
considered metric spaces are finite and all of the hyperplanes in
$X$ are isomorphic (with uniform bound on all such isomorphisms)
we may assume that images of all of $F_n$ are contained in a fixed
hyperplane $H$ in $X$. We may assume that $H$ is a kernel of a
functional $x^*\in X^*$, $||x^*||=1$ which attains its norm on
some vector $x\in X$, $||x||=1$. We may assume also that
$F_n(O_n)=0$ for each $n$. Now we modify $F_n$ by shifting each of
the $F_n(S_n)$, $n\ge 2$, by $\lambda_n x$, where $\lambda_n$ is
the sum of the lengths of the paths $P_1,\dots,P_{n-1}$, and we
map vertices of paths $P_1,\dots,P_{n-1}$ to the one-dimensional
space spanned by $x$ in such a way that the distance between any
two adjacent vertices is $1$. The Lipschitz constant of the
obtained map $F:S\to X$ is equal to $1$ since it is does not
increase distance between endpoints of edges.\medskip

To estimate $\lip(F^{-1})$ consider any two vertices $u$ and $v$
in $S$. It is clear that if $u,v$ are in the same $S_n$, then
$d_S(u,v)\le \lip((F_n)^{-1})||F(u)-F(v)||$. It is also clear that
the embedding $F$ is an isometry on the union of the paths
$\{P_n\}$. Therefore it remains to consider the cases (i) $u$ and
$v$ are in different $S_n$; (ii) $u$ is in one of the $\{S_n\}$
and $v$ is in one of the $\{P_n\}$.
\medskip

Case (i). Let $u\in S_n$, $v\in S_m$, $n<m$. In this case
$d_S(u,v)=d_{S_n}(u,O_n)+($the sum of lengths of paths
$P_n,\dots,P_{m-1})+d_{S_m}(O_m,v)$, and $||F(u)-F(v)||\ge
|x^*(F(u)-F(v))|=($the sum of lengths of paths
$P_n,\dots,P_{m-1})$. Recalling the way in which we choose the
lengths of $\{P_n\}$, we get the desired conclusion.

Case (ii). Let $u\in S_n$, $v\in P_m$.  In this case we have to
consider two subcases: $d_S(v,O_n)\ge
\frac1{2\lip((F_n)^{-1})}d_S(u,O_n)$ and $d_S(v,O_n)<
\frac1{2\lip((F_n)^{-1})}d_S(u,O_n)$. In the first subcase we have
\[||F(u)-F(v)||\ge
|x^*(F(u)-F(v))|=d_S(v,O_n)\ge\frac1{4\lip((F_n)^{-1})}d(u,v).\]

In the second subcase we get
\[\begin{split}||F(u)-F(v)||&\ge ||F(u)-F(O_n)||-||F(v)-F(O_n)||\\&\ge
\frac1{\lip((F_n)^{-1})}d_S(u,O_n)-d_S(v,O_n)\\&>\frac1{2\lip((F_n)^{-1})}d_S(u,O_n)\\&\ge
\frac1{4\lip((F_n)^{-1})}d_S(u,v)\qedhere
\end{split}
\]
\end{proof}

\end{large}

\renewcommand{\refname}{\section{References}}

\begin{small}

\end{small}

\end{document}